\documentclass{amsart}
\numberwithin{equation}{section}

\newtheorem{thm}{Theorem}[section]
\newtheorem{lem}[thm]{Lemma}
\newtheorem{cor}[thm]{Corollary}

\newtheorem{defin}[thm]{Definition}

\newtheorem{remark}[thm]{Remark}

\begin{document}

\begin{center}
\textbf{{\large {\ Inverse problem for fractional order subdiffusion equation }}}\\[0pt]
\medskip \textbf{Marjona Shakarova$^{1}$}\\[0pt]
\textit{shakarova2104@gmail.com\\[0pt]}
\medskip \textit{\ $^{1}$ Institute of Mathematics, Academy of Science of Uzbekistan}

\end{center}

\textbf{Abstract}: The study examines the inverse problem of finding the appropriate right-hand side for the subdiffusion equation with the Caputo fractional derivative in a Hilbert space represented by $H$. The right-hand side of the equation has the form $g(t)f$ and an element $f\in H$ is unknown. If the sign of $g(t)$ is a constant, then the existence and uniqueness of the solution is proved. When $g(t)$ changes sign, then in some cases, the existence and uniqueness of the solution is proved, in other cases, we found the necessary and sufficient condition for a solution to exist. Obviously, we need an extra condition to solve this inverse problem. We take the additional condition in the form $\int\limits_0^Tu(t)dt=\psi$. Here $\psi $ is a given element, of $H$. 


\section{Introduction}

Suppose that $H$ is a separable Hilbert space  with the scalar product $(\cdot, \cdot)$, and let $A$ be an operator on $H$, with a domain of definition  $D(A)$, satisfying the following conditions:

1) $A=A^*$, where $A^*$ denotes the adjoint operator of $A$,

2) $(Ah,h)\geq C(h,h)$, $h \in D(A)$, for some $C>0$.

Assume that $A$ has a complete system of orthonormal eigenfunctions ${v_k}$ in $H$ and a countable set of positive eigenvalues $\lambda_k$. It is assumed that the eigenvalues are ordered such that $0<\lambda_1\leq\lambda_2\leq \cdots\rightarrow +\infty$.

Let $C((a,b);H)$ stand for a set of continuous functions $u(t)$ of $t\in (a,b)$ with values in $H$.

$D_t^\rho y(t)$ is the Caputo fractional derivative defined as (see, \cite{Pskhu}):
\begin{equation*}
D_t^\rho y(t)=\frac{Y(t)}{\Gamma
(1-\rho)}, \quad Y(t)=\int\limits_0^t \frac{\frac{d}{d\xi}y(\xi)}{(t-\xi)^{\rho}}d\xi,
\quad t>0,
\end{equation*} 
where $\Gamma(\rho)$ is
Euler's gamma function. 

We note that the fractional derivative and the regular classical derivative of the first order are equivalent if $\rho=1$: $D_t h(t)= \frac{d}{dt} h(t)$.

\textbf{Problem}. We study the inverse problem of finding functions $\{u(t), f\}$ that satisfy the following subdiffusion equation
\begin{equation}\label{prob1}
D_t^\rho u(t) + Au(t) =g(t)f,\quad \rho\in(0,1],\quad t\in (0, T],
\end{equation}
with the initial
\begin{equation}\label{in.c}
    u(0)=\varphi,
\end{equation}
and the additional conditions
\begin{equation}\label{ad}
\int\limits_0^T u(t)dt=\psi. 
\end{equation}
Here $g(t)\in C[0,T]$ is a given function and $\varphi,\psi\in H$ are known elements.

 The solution of the inverse problem will involve examining the Cauchy problem for different types of differential equations. In this context, when we refer to the solution of the problem, we specifically mean the classical solution. This implies that all the derivatives and functions involved in the equation are assumed to be continuous with respect to the variable $t$. As an example, present the definition of the solution of the inverse problem (\ref{prob1})-(\ref{ad}).

\begin{defin}\label{def1} A pair of functions $\{u(t), f\}$ with the properties $D_t^\rho u(t), Au(t)\in C((0,T]; H)$, $u(t)\in C([0,T];H)$, $f\in H$ satisfying conditions (\ref{prob1})-(\ref{ad}) is called \textbf{the solution} of the inverse problem.
\end{defin}

Recently, inverse problems related to integer or fractional order differential equations have received more attention among researchers.

Most research on source function determination focuses on specific processes such as $F = g(t)f(x)$, where either $g (t)$ or $f(x)$ is unknown. Inverse problems of finding the function $g(t)$ have been studied, for example, in \cite{Hand1}-\cite {Ash2}). When $f(x)$ is unknown and $g(t)\equiv1$, the inverse problems have been studied by many authors (see \cite{Fur}-\cite{4}). In this work, we focus on the problem of determining the function $f(x)$, when $g(t)\not\equiv1$. Similar problems for the diffusion equation are studied in the well-known monographs of S.Kabanikhin \cite{Kab1} and the papers \cite{Pr}-\cite{FN3}. As for the subdiffusion equation, such inverse problems are studied in papers \cite{MS}-\cite{AshM2}. Let us mention some of the results obtained for the diffusion and subdiffusion equations.

We briefly note some known results on inverse problems for the diffusion equation. A.I. Prilepko and A.B. Kostin \cite{Pr} presented the elliptic part of the diffusion equation as a second-order differential expression. The authors consider both a non-self-adjoint and a self-adjoint elliptic part. They established a criterion of uniqueness of the generalized solution of the inverse problem when elliptic part is self-adjoint. Note, that here the additional condition is taken in an integral form.
Unlike to the paper \cite{Pr}, in papers \cite{Sab}, \cite{Sab2} the problem of finding the function $f(x)$ for the diffusion equation was studied using the additional condition $u(x,t_0)=\psi$. Some authors set the additional condition as $t_0=T$ (see, e.g. \cite{Orl}, \cite{Tix} for classical diffusion equations and for subdiffusion equations see \cite{MS}, \cite{MS1}). 

An inverse problem similar to (\ref{prob1})-(\ref{ad}) for various operators $A$ and with the Caputo and Riemann-Liouville derivatives are considered in \cite{FN}-\cite{FN2}, and in \cite{FN} the fractional derivative is taken in the sense of Caputo and in \cite{FN2} in the sense of Riemann-Liouville. In \cite{FN}, the criteria for the uniqueness of the solution of the inverse problem are found. And in work \cite{FN2} the question of the correctness of the inverse problem by operator methods was studied.

    In the paper \cite{AshM2} of the researchers analyzed subdiffusion equation with the Caputo derivative in which the Laplace operator forms the elliptic part. This paper focused on forward and inverse problems for the subdiffusion equation. The authors of the study proved the uniqueness and existence of the solution of the inverse problem, if the function $g(t)$ preserves its sign. Moreover, if the function $g(t)$ changes sign, a necessary and sufficient condition for the existence of a classical solution was found, and all solutions of the inverse problem were constructed using the classical Fourier method. It should be noted that all the findings presented in this paper for the case where $g(t)$ changes its sign are also new for the classical diffusion equation. Finally, we will use some original ideas from this work to solve our inverse problem.
\ \ \ \

We introduce the power of operator $ A $ with domain $$
D(A^\tau)=\{h\in H:  \sum\limits_{k=1}^\infty \lambda_k^{2\tau}
|h_k|^2 < \infty\},
$$ 
acting in $H$ according to the rule:
$$
A^\tau h= \sum\limits_{k=1}^\infty \lambda_k^\tau h_k v_k.
$$
Here $ \tau $ is an arbitrary real number and $h_k=(h,v_k)$ are the Fourier coefficients of a element $h \in H$.

For elements $h,g \in D(A^\tau)$ we introduce the scalar product:
\[
(h,g)_\tau=\sum\limits_{k=1}^\infty \lambda_k^{2\tau} h_k \overline {g_k} =
(A^\tau h,A^\tau g)
\]
and together with this norm $D(A^\tau)$ turns into a Hilbert
space.

\section{Preliminaries}

The problem of finding the function $u(t)$ satisfying subdiffusion equation (\ref{prob1}) with initial condition (\ref{in.c}) is also called \emph{the forward problem}. The forward problem is well-studied in the literature, and the existence and uniqueness of the solution have been proved in various works, including \cite{AshM2}, \cite{AshM}. These works provide important theoretical foundations for studying the inverse problem. We mention the solution of the forward problem to solve the inverse problem (\ref{prob1})-(\ref{ad}) we are studying:
\begin{equation}\label{fp}
       u(t)=\sum\limits_{k=1}^{\infty}\left[\varphi_k E_{\rho,1}(-\lambda_k t^\rho)+ f_k\int\limits_0^t (t-\eta)^{\rho-1} E_{\rho, \rho} (-\lambda_k  (t-\eta)^\rho) g(\eta)d\eta\right]v_k, 
\end{equation}
    where $\varphi_k$, $f_k$ are the Fourier coefficients of functions $\varphi$, $f$, respectively and 
   \[
E_{\rho, \mu}(z)= \sum\limits_{n=0}^\infty \frac{z^n}{\Gamma(\rho
n+\mu)} \quad 0 < \rho < 1, \quad z,\mu\in \mathbb{C}  
   \] 
is called the Mittag-Leffler function with two-parameters (see, \cite{Dzh66}, p. 133).

To find the unknowns $\{u(t), f\}$ of inverse problem (\ref{prob1})-(\ref{ad}), we apply additional condition (\ref{ad}) to equality (\ref{fp}). Then obtain the following equality:
\[
\sum\limits_{k=1}^{\infty}\left[\varphi_k \int\limits_0^T E_{\rho,1}(-\lambda_k t^\rho)dt+ f_k \int\limits_0^T\int\limits_0^t (t-\eta)^{\rho-1} E_{\rho, \rho}(-\lambda_k(t-\eta)^\rho) g(\eta) d\eta dt \right]v_k=\psi.
\]    

Now we introduce the following lemmas:

\begin{lem}\label{lem1} Let $\rho> 0$, then the following equality is hold:
 \[ \int\limits_0^TE_{\rho,1}(-\lambda_k t^\rho)dt =TE_{\rho,2}(-\lambda_k T^\rho).\]
\end{lem} 
\begin{proof} The proof of this lemma follows from the following equality (see, \cite{Gor}, formula (4.4.4), p. 61):
    \begin{equation}\label{MLI}
			\int\limits_0^t \eta^{\beta-1}E_{\rho,\beta}(\lambda\eta^\rho)d\eta=t^{\beta} E_{\rho,\beta+1}(\lambda t^\rho), \quad \rho>0, \quad\beta > 0, \quad \lambda \in C,
		\end{equation}
\end{proof}
\begin{lem}\label{lem2} Let $\rho> 0$, then 
\begin{equation*} \int\limits_0^T\int\limits_0^{t} (t-\eta)^{\rho-1} E_{\rho, \rho} (-\lambda_k  (t-\eta)^\rho ) g(\eta) d\eta dt =  \int\limits_0^T g(\eta) (T-\eta)^{\rho} E_{\rho, \rho+1} (-\lambda_k  (T-\eta)^\rho )  d\eta.
\end{equation*}
 \begin{proof}
  By calculating the double integral, we obtain the following equality:
\begin{equation}\label{in2}
 \int\limits_0^T\int\limits_0^{t} (t-\eta)^{\rho-1} E_{\rho, \rho} (-\lambda_k  (t-\eta)^\rho ) g(\eta) d\eta dt
\end{equation}
\begin{equation*}
 =\int\limits_0^T g(\eta)d\eta \int\limits_\eta^{T} (t-\eta)^{\rho-1} E_{\rho, \rho} (-\lambda_k  (t-\eta)^\rho ) dt=\int\limits_0^T g(\eta)d\eta \int\limits_0^{T-\eta} s^{\rho-1} E_{\rho, \rho} (-\lambda_k  s^\rho ) ds  
\end{equation*}
\begin{equation*}
 =\int\limits_0^T g(\eta)d\eta \int\limits_0^{T-\eta} s^{\rho-1} E_{\rho, \rho} (-\lambda_k  s^\rho ) ds. 
\end{equation*}
Due to equality (\ref{lem1}), (\ref{in2}) is equal to the following integral:
\begin{equation*}
 \int\limits_0^T g(\eta) (T-\eta)^{\rho} E_{\rho, \rho+1} (-\lambda_k  (T-\eta)^\rho )  d\eta.
\end{equation*}  
 \end{proof}
\end{lem}
According to Lemma \ref{lem1} and Lemma \ref{lem2}, we have the following equality:
\begin{equation*}
 \sum\limits_{k=1}^{\infty}\left[\varphi_k TE_{\rho,2}(-\lambda_k T^\rho )dt+ f_k \int\limits_0^T (t-\eta)^{\rho} E_{\rho, \rho+1} (-\lambda_k  (t-\eta)^\rho ) g(\eta) d\eta dt \right ]v_k=\psi.     
\end{equation*}
If we expand the function $\psi$ into the Fourier series according to the system $\{ v_k\}$ and equate the Fourier coefficients, then we have the following equality:

\begin{equation}\label{EqFor_fk1}
f_k p_{k,\rho}(T)=
   \psi_k - \varphi_k TE_{\rho,2}(-\lambda_k T^\rho).
\end{equation}
 where
\begin{equation*} 
p_{k,\rho}(T)=\int\limits_0^T g(\eta) (T-\eta)^{\rho} E_{\rho, \rho+1} (-\lambda_k  (T-\eta)^\rho )  d\eta.
\end{equation*} 

According to the idea of the authors of \cite{AshM2}, we divide $\mathbb{N}$ into two sets, i.e. $N=B_\rho \cup B_{0,\rho} $. Here, $\mathbb{N}$ represents the set of all natural numbers. The sets $B_\rho$ and $B_{0,\rho}$ are defined as follows:

1) If the function $p_{k,\rho}(T)\neq 0$, then $k\in B_\rho$,

2) Alternatively, if the function $p_{k,\rho}(T)=0$, then $k\in B_{0,\rho}$.

It is obvious, if $g(t)$ is a sign-preserving function, then $p_{k,\rho}(T)\neq 0$. Therefore, in this case the set $B_{0,\rho}$ is empty and $B_\rho=\mathbb{N}$. 

Equation (\ref{EqFor_fk1}) provides us with a means to determine $f_k$. It can be observed that the criterion for the uniqueness of the solution to the inverse problem associated with the diffusion and subdiffusion equations can be expressed as follows:
\[
p_{k,\rho}(T)\neq 0.
\]
According to this criterion, for the solution to be unique, it is necessary that the expression $p_{k,\rho}(T)$ does not equal zero.

To establish two-sided estimates for $p_{k,\rho}(T)$, let's consider the case where the function $g(t)$ does not change sign. In this case, the set $B_{0,\rho}$ is empty. Then the following lemma holds.
 
\begin{lem}\label{invvv1}Let $\rho\in (0,1]$, $g(t)\in C[0,T]$ and $g(t)\neq 0$, $t\in [0,T]$. Then there are constants $C_0,C_1>0$, depending on $T$, such that for all $k$:
		\begin{equation*}
		\frac{C_0}{\lambda_k}\leq |p_{k,\rho}(T)|\leq\frac{C_1}{\lambda_k}.    
		\end{equation*}
	\end{lem}
\begin{proof}
Let $\rho=1$. By integrating by parts and  the mean value theorem, we obtain
    \[
    p_{k,1}(T)=\frac{1}{\lambda_k}\int\limits_0^{T} (1-e^{-\lambda_k  s}) g(T-s)ds  =
    \]
    \[
    =\frac{g(\xi_k)}{\lambda_k} \bigg[{T}-\frac{1}{\lambda_k} (1-  e^{-\lambda_k T})\bigg], \quad \xi_k\in [0, T].    
\]

By virtue of the Weierstrass theorem, we have $|g(t)|\geq g_0=const >0$. Then we can establish the lower and upper bounds as follows:

\[ \frac{g_0c_0}{\lambda_k}\leq|p_{k,1}(T)|\leq\frac{\max\limits_{0\leq\xi \leq T}|g(\xi)|T}{\lambda_k}.\]
Let $\rho \in (0,1)$. Apply the mean value theorem and equality (\ref{MLI}) to obtain
  \[
  |p_{k,\rho}(T)| =\bigg|\int\limits _0^{T} \eta^{\rho} E_{\rho, \rho+1} (-\lambda_k  \eta^\rho)g(T-\eta)d\eta\bigg|=
  \]
  \[=|g(\xi_k)| T^{\rho+1} E_{\rho, \rho+2} (-\lambda_k T^\rho ) , \quad \xi_k\in[0,T].\]  
  
  Therefore, using the asymptotic estimate of the Mittag-Leffler function (see, \cite{Dzh66}, p. 134)
  \begin{equation}\label{MLA}
   E_{\rho, \mu}(-t)=\frac{t^{-1}}{\Gamma(\mu-\rho)}+O(t^{-2})    
  \end{equation}
  and the estimate $|g(t)|\geq g_0$ one has 
  \[
  |p_{k,\rho}(T)|={|g(\xi_k)|{T^{\rho+1}}}\bigg(({T^\rho\lambda_k})^{-1}+O({\lambda_kT^{\rho}})^{-2}\bigg)\geq  \frac{C_0}{\lambda_k}.
  \]  
  Finally, according to the estimate of the Mittag-Leffler function (see, \cite{Dzh66}, p. 136) 
\begin{equation}\label{Ml}
|E_{\rho, \mu}(-t)|\leq \frac{C}{1+t}, \quad t\geq0
\end{equation}
(where constant $C$ does not depend on $t$ and $\mu$), we have
  \[
  |p_{k,\rho}(T)|\leq C\frac{|g(\xi_k)|T^{\rho+1}}{1+\lambda_k T^\rho} \leq C\frac{ \max\limits_{0\leq\xi \leq T}|g(\xi)|T}{\lambda_k}\leq \frac{C_1}{\lambda_k}.
  \]
  
\end{proof}

Now consider the case when $g(t)$ changes sign. Then the function $p_{k,\rho}(T)$ can become zero, and as a result, the set $B_{0,\rho}$ may turn out to be non-empty. In the case where the sign of $g(t)$ is a variable function, we will present the following lemma.
\begin{lem}\label{lemmaSub} Let $\rho\in (0,1]$, $g(t)\in C^1[0, T]$ and $g(0)\neq 0$. 
Then there exist numbers $m_0>0$ and $k_0$ such that, for all  $T\leq m_0$ and $k\geq k_0$, the following estimates hold:
		\begin{equation}\label{estimateSub}
		\frac{C_0}{\lambda_k}\leq |p_{k,\rho}(T)|\leq\frac{C_1}{\lambda_k}.
		\end{equation}
  where constants $C_0$ and $C_1>0$ depend on $m_0$ and $k_0$.
	\end{lem}
 \begin{proof}
Let $\rho=1$. By integrating by parts and  the mean value theorem, we get
    \[
    p_{k,1}(T)=\frac{1}{\lambda_k}\int\limits_0^{T} (1-e^{-\lambda_k  s}) g(T-s)ds  =\frac{1}{\lambda_k}\bigg[g(T-s)(s+\frac{e^{-\lambda_k  s}}{\lambda_k})\bigg|^{T}_0\]
    \[+\int\limits_0^{T} (s+\frac{e^{-\lambda_k  s}}{\lambda_k}) g'(T-s)ds\bigg]
    \]
    \[
    =\frac{g(0)}{\lambda_k} \bigg(T+\frac{e^{-\lambda_k T}}{\lambda_k}\bigg)-\frac{g(T)}{\lambda_k^2} +  \frac{g'(\xi_k)}{\lambda_k}\big[\frac{T^2}{2}-\frac{1}{\lambda_k^2}(1-e^{-\lambda_k T}
)\big], \quad \xi_k\in [0, T].    
\]
since $k\geq k_0$
\[
    |p_{k,1}(T)|\geq\bigg|\frac{g(0)}{\lambda_k}T-\frac{g(T)}{\lambda_k^2}\bigg|.    
\]
If $g(0)\neq 0$, then for large $k$ we can conclude that there exists a constant $C_0$ such that the lower bound in the estimate holds.

To establish the upper estimate, we utilize the boundedness of the function $g(t)$.
 
  Let $\rho\in (0,1)$. Using equality (\ref{MLI}) we integrate by parts, then apply the mean value theorem. Then we have
    \[  
    p_{k,\rho}(T)=\int\limits_0^{T}g(T-s) s^{\rho} E_{\rho, \rho+1} (-\lambda_k  s^\rho )  ds=\int\limits_0^{T}g(T-s) d\big[ s^{\rho+1} E_{\rho, \rho+2} (-\lambda_k  s^\rho ) \big] =
    \]
    \[
    =g(T-s)  s^{\rho+1} E_{\rho, \rho+2} (-\lambda_k  s^\rho )\bigg|^{T}_0+\int\limits_0^{T}g'(T-s)  s^{\rho+1} E_{\rho, \rho+2} (-\lambda_k  s^\rho )ds=
    \]
    \[
    =g(0)\,  T^{\rho+1} \,E_{\rho, \rho+2} (-\lambda_k  T^\rho)+ g'(\xi_k) \int\limits_0^{T} s^{\rho+1} E_{\rho, \rho+2} (-\lambda_k  s^\rho )ds, \quad \xi_k\in [0, T].
    \]
    For the last integral formula (\ref{MLI}) implies
\[
    \int\limits_0^{T} s^{\rho+1} E_{\rho, \rho+2} (-\lambda_k  s^\rho )ds=T^{\rho+2} E_{\rho, \rho+3}(-\lambda_k T^\rho).
\]
    Apply the asymptotic estimate (\ref{MLA}) to get
\[
p_{k,\rho}(T)=\frac{g(0)T}{\lambda_k} +\frac{g'(\xi_k)}{\lambda_k} T^2 + O\bigg(\frac{1}{(\lambda_k T^\rho)^2}\bigg).
\]
If $g(0)\neq 0$, we can infer that for sufficiently small $T$ and sufficiently large $k$, the required lower estimate holds. Additionally, this implies the required upper bound as well.
\end{proof}  
\begin{cor}\label{Krho}If conditions of Lemma \ref{lemmaSub} are  satisfied, then estimate (\ref{estimateSub}) holds for suffuciently small $T$ and $k\in B_{\rho}$.
 \end{cor}

\begin{cor}\label{Krho1}If conditions of Lemma \ref{lemmaSub} are  satisfied and $T$ is sufficiently small, then  set $B_{0,\rho}$ has a finite number elements.
\end{cor}
\begin{remark}
    In the paper \cite{AshM2}, a lemma similar to the above lemma was proved for the diffusion and subdiffusion equations. In this paper $g(t_0)\neq 0$ and $g(0)\neq 0$ were for $\rho=1$ and $\rho\in (0,1)$, respectively. In this paper, in cases where $ \rho=1 $ and $ \rho\in(0.1)$, conditions $ g (t_0) \neq 0 $ and $ g(0) \neq 0$ for function $g(t)$ were found, respectively. However, in our lemma, for the diffusion and subdiffusion equations, for function $g(t)$ one has the same condition, i.e. $g(0)\neq 0$.
\end{remark}
\section{The solution of problem (\ref{prob1})-(\ref{ad})}
If $g(t)$ is a sign-preserving function, then the following theorem holds. 
\begin{thm}\label{thmNotChange} 
Let $\rho\in (0,1]$, $\varphi \in H$, $\psi \in D(A)$, $g(t)\in C[0,T]$ and $g(t)\neq 0$, $t\in [0,T]$. Then there exists a unique solution of the inverse problem (\ref{prob1})-(\ref{ad}):
		\begin{equation*}
		 f=\sum\limits_{k=1}^\infty \frac{1}{p_{k,\rho}(T)}\left[ \psi_k-\varphi_k TE_{\rho,2}(-\lambda_k  T^\rho)\right]v_k,   
		\end{equation*}
	\begin{equation*}
	u(t)=\sum\limits_{k=1}^\infty \left[\varphi_k E_{\rho,1}(-\lambda_kt^\rho)+ \frac{p_{k,\rho}(t)}{p_{k,\rho}(T)}\left[ \psi_k-\varphi_k TE_{\rho,2}(-\lambda_k  T^\rho)\right]\right]v_k.    
	\end{equation*}

\end{thm}
Now we form the following corresponding result for the case when sign of function $g(t)$ has changed.
\begin{thm}\label{thm2} Let $\rho\in (0,1]$, $\varphi \in H$, $\psi \in D(A)$, $g(t)\in C^1[0,T]$. Further, we will assume that the conditions of Lemma \ref{lemmaSub} are satisfied and $T$ is sufficiently small.
		
		1) If set $B_{0,\rho}$ is empty, for all $k$, then there exists a unique solution of the inverse problem (\ref{prob1})-(\ref{ad}):
		\begin{equation*}
		 f=\sum\limits_{k=1}^\infty \frac{1}{p_{k,\rho}(T)}\left[ \psi_k-\varphi_k TE_{\rho,2}(-\lambda_k  T^\rho)\right]v_k,   
		\end{equation*}
	\begin{equation*}
	u(t)=\sum\limits_{k=1}^\infty \left[\varphi_k E_{\rho,1}(-\lambda_kt^\rho)+ \frac{p_{k,\rho}(t)}{p_{k,\rho}(T)}\left[ \psi_k-\varphi_k TE_{\rho,2}(-\lambda_k  T^\rho)\right]\right]v_k.    
	\end{equation*}		
		2) If set $B_{0,\rho}$ is not empty, then for the existence of a solution to the inverse problem, it is necessary and  sufficient that the following  conditions
		\begin{equation}\label{ortogonal}
			\psi_k=\varphi_k T E_{\rho,2}(-\lambda_kT^\rho),\quad k\in B_{0,\rho}
		\end{equation}
		 be satisfied. In this case, the solution to the problem (\ref{prob1})-(\ref{ad}) exists, but is not unique:
		\begin{equation}\label{f_2}
		   f=\sum\limits_{k\in B_{\rho}} \frac{1}{p_k(T)}\left[ \psi_k-\varphi_k TE_{\rho,2}(-\lambda_k  T^\rho)\right]v_k+\sum\limits_{k \in B_{0,\rho}} f_k v_k, 
		\end{equation}
		\begin{equation}\label{u_2}
		u(t)=\sum\limits_{k=1}^\infty\big[\varphi_k E_{\rho,1}(-\lambda_kt^\rho)+f_k\big]v_k,    
		\end{equation}
where $f_k$, $k\in B_{0,\rho}$, are arbitrary real numbers.
 \end{thm}
As mentioned earlier, Theorem \ref{thmNotChange} for the diffusion equation ($\rho=1$) with the additional condition $u(x,t_0)=\psi$ has only been proven in the cases where $\Omega$ is an interval on $\mathbb{R}$ (see, \cite{Sab}) or a rectangle in $\mathbb{R}^2$ (see, \cite{Sab2}). The inverse problem (\ref{prob1})-(\ref{in.c}) with the same additional condition, considering both the cases when the function $g(t)$ changes sign and when it does not change sign, has been addressed in the work of Ashurov et al. (see, \cite{AshM2}). However, the theorems we have presented above, for both the diffusion and subdiffusion equations, involve an integral additional condition (\ref{ad}). It is worth noting that these theorems are also novel for diffusion equations. Besides, we must also note that, unlike the paper \cite{AshM2}, in the theorems we have proven, the condition is given not to point $t_0$, but to the boundary of the domain i.e $T$.

\textbf{Proof of Theorem \ref{thmNotChange}.}
  Since $p_{k,\rho}(T)\neq 0 $ for all $ k \in \mathbb{N}$, then we get the following equations from (\ref{EqFor_fk1}):
\begin{equation*}
f_k=\frac{1}{p_{k,\rho}(T)}\left[ \psi_k-\varphi_kTE_{\rho,2}(-\lambda_k  T^\rho)\right].
\end{equation*}

From these $f_k$ are Fourier coefficients of the unknown $f$, has the form:
\begin{equation}\label{inv10}
f=\sum\limits_{k=1}^\infty \frac{1}{p_{k,\rho}(T)}\left[ \psi_k-\varphi_kTE_{\rho,2}(-\lambda_k  T^\rho)\right]v_k.
\end{equation}

Let us prove the uniformly convergence of this series.

Let $F_j$ be the partial sum of series (\ref{inv10}):
\[F_j=\sum\limits_{k=1}^j \frac{1}{p_{k,\rho}(T)}\left[ \psi_k-\varphi_kTE_{\rho,2}(-\lambda_k  T^\rho)\right]v_k= F_{j,1}+F_{j,2}.\]
Then we show that every series $F_{j,1}$ and $F_{j,2}$ are absolutely and uniformly convergent.

First we estimate of the series $F_{j,1}$. For this, applying Parseval's equality, we arrive at:
\begin{equation*}
  ||F_{j,1}||^2=\bigg|\bigg|\sum\limits_{k=1}^j \frac{\psi_k}{p_{k,\rho}(T)}v_k\bigg|\bigg|^2 \leq \sum\limits_{k=1}^j \frac{1}{|p_{k,\rho}(T)|^2} |\psi_k|^2\leq C\sum\limits_{k=1}^j\lambda_k^2|\psi_k|^2 = C ||\psi||^2_1.  
\end{equation*}

Now, we estimate of the series $F_{j,2}$. According to Parseval's equality and estimate (\ref{Ml}), we have:
\begin{equation*}
 ||F_{j,2}||^2=\bigg|\bigg|\sum\limits_{k=1}^j \frac{\varphi_kTE_{\rho,2}(-\lambda_k  T^\rho)}{p_{k,\rho}(T)}v_k\bigg|\bigg|^2 \leq \sum\limits_{k=1}^j \left|\frac{TE_{\rho,2}(-\lambda_k  T^\rho)}{p_{k,\rho}(T)}\right|^2 |\varphi_k |^2\leq  C ||\varphi||^2. 
 \end{equation*}

Thus, if $\varphi \in H$, $\psi \in D(A)$, then from estimates of $F_{i,j}$ we obtain $f \in H$.

If $f\in H$ is known function, then we obtained the following equality for function $u(t)$: 
\begin{equation}\label{inv11}
u(t)=\sum\limits_{k=1}^\infty \left[\varphi_k E_{\rho,1}(-\lambda_kt^\rho)+ \frac{p_{k,\rho}(t)}{p_{k,\rho}(T)}\left[ \psi_k-\varphi_k TE_{\rho,2}(-\lambda_k  T^\rho)\right]\right]v_k.
\end{equation}

From this equality, we have the following form for Fourier coefficients $u_k(t)$ of function $u(t)$:
\begin{equation*}
u_k(t)=\varphi_k E_{\rho,1}(-\lambda_kt^\rho)+ \frac{p_{k,\rho}(t)}{p_{k,\rho}(T)}\left[ \psi_k-\varphi_k TE_{\rho,2}(-\lambda_k  T^\rho)\right].
\end{equation*}

Now we need to show that function $u(t)$ is a solution of inverse problem (\ref{prob1})-(\ref{ad}). Fulfillment of the conditions of Definition \ref{def1} for function $u(t)$, defined by the series (\ref{inv11}) is proved in exactly the same way as the solution of the forward problem (\ref{prob1}). As we noted above, the solution to the forward problem was proved in papers \cite{AshM2}, \cite{AshM}.

The uniqueness of the solution was proved in paper \cite{AshM2}. Therefore, we briefly cite the proof of the uniqueness.

To prove the uniqueness of the solution, assume the opposite, that is, there are two different solutions $\{u_1, f_1\}$ and $\{u_2, f_2\}$ satisfying the inverse problem (\ref{prob1} )-(\ref{ad}). We must show that $u\equiv u_1-u_2 \equiv 0$, $f\equiv f_1-f_2\equiv 0$. For $\{u, f\} $ we have the following problem::
 \begin{equation}\label{prob20}
\left\{
\begin{aligned}
& D_t^\rho u(t)+A u(t) =g(t)f,\quad t\in (0,T],\\
&u(0)=0, \\ 
&\int\limits_0^T u(t)dt=0.\\
\end{aligned}
\right.
\end{equation}
We take any solution $\{u,f\}$ and define $u_k=(u,v_k)$ and $f_k=(f,v_k)$. Then, due to the self-adjointness of the operator $A$, we obtain
\[
D_t^\rho u_k(t)= (D_t^\rho u, v_k)= -(A u, v_k)+f_k g(t)=-( u,A v_k)+f_k g(t)=-\lambda_k u_k(t)+f_k g(t).
\]
Therefore, for $u_k$ we obtain the Cauchy problem

\[
D_t^\rho u_k(t)+\lambda_k u_k(t) =f_kg(t),\quad t>0,\quad u_k(0)=0,
\]
and the additional condition
\[
 \int\limits_0^T u_k(t)dt=0.
\]
If $f_k$ is known, then the unique solution of the Cauchy problem has the form
\[
u_k(t)= f_k\int\limits_0^t \eta^{\rho-1} E_{\rho, \rho} (-\lambda_k \eta^\rho) g(t-\eta) d\eta.
\]
Apply the additional condition to get
\[
\int\limits_0^T u_k(t)dt= f_k\int\limits_0^T g(\eta) (T-\eta)^{\rho} E_{\rho, \rho+1} (-\lambda_k  (T-\eta)^\rho )  d\eta=f_kp_{k,\rho}(T)=0.
\]
Since $p_{k,\rho}(T) \neq 0$ for all $k \in \mathbb{N} $, then due to completeness of the set of eigenfunctions $\{v_k\}$ in $H$, we finally have  $f\equiv 0$ and  $u(t)\equiv0$.
 \hfill$\Box$ \
 
 We will now proceed with the proof of Theorem \ref{thm2}.

\textbf{Proof of Theorem \ref{thm2}.}
We will consider the proof of the theorem for cases where the set $B_{0,\rho}$ is empty and non-empty.

When $p_{k,\rho}(T)\neq 0$ for all $k$, we can prove the existence and uniqueness of the solution of functions ${\{u(t),f\}}$ in the same way as in Theorem \ref{thmNotChange}.

Next, we consider the case where $B_{0,\rho}$ is not an empty set. If $k\in B_{\rho}$, we can use Lemma \ref{lemmaSub} to prove the first part of equalities (\ref{f_2})-(\ref{u_2}) in the same way as the existence of a solution was proved in Theorem \ref{thmNotChange}.
However, when $k\in B_{0,\rho}$, the solution of equation (\ref{EqFor_fk1}) with respect to $f_k$ exists if and only if the extra conditions (\ref{ortogonal}) are satisfied. The solution of equation (\ref{EqFor_fk1}) in this case can be arbitrary numbers $f_k$. 

Instead of condition (\ref{ortogonal}), according to $0<E_{\rho,2} (-t)<1$, (see \cite{Gor}, p. 47) we can use the orthogonality conditions which are easy to verify:

\[\varphi_k=(\varphi, v_k)=0, \quad \psi_k=(\psi, v_k)=0, \quad k\in B_{0,\rho}.\]
 
\hfill$\Box$ \
\section*{Acknowledgements}
 The author is grateful to R.R. Ashurov for discussions of these results. The author acknowledges financial support from the  Ministry of Innovative Development of the Republic of Uzbekistan, Grant No F-FA-2021-424.



\begin{thebibliography}{99}
 \normalsize


\bibitem{Hand1} Y. Liu, Z. Li, M. Yamamoto. Inverse problems of determining sources of the fractional partial differential equations, Handbook of Fractional Calculus with Appl. J.A.T. Marchado Ed.~De Gruyter. 2, 411-430 (2019).

\bibitem {Yama11}  K. Sakamoto, M. Yamamoto. Initial value boundary value problems for fractional diffusion-wave equations and applications to some inverse problems, J. Math. Anal. Appl. {\bf 382}, 426-447 (2011).

\bibitem{Ash1}  R. Ashurov, M. Shakarova. Time-dependent source identification problem for fractional Schr\"odinger type equations, Lobachevskii Journal of Mathematics. {\bf 42}:3, 517-525 (2022).

\bibitem{Ash2} R. Ashurov, M. Shakarova. Time-dependent source identification problem for a fractional Schrodinger equation with the Riemann-Liouville derivative, https: arxiv.org/abs/2205.03407. 6 may 2022.

\bibitem{Fur} K. Furati, O. Iyiola, M. Kirane. An inverse problem for a generalized fractional diffusion, Applied Mathematics and Computation. {\bf 249}, 24-31 (2014).

\bibitem{15} M. Kirane, A. Malik. Determination of an unknown source term and the temperature distribution for the linear heat equation involving fractional derivative in time, Applied Mathematics and Computation. {\bf 218}, 163-170 (2011).

\bibitem{16}   M. Kirane, B. Samet, B. Torebek. Determination of an unknown source term and the temperature distribution for the subdiffusion equation at the initial and final data, Electronic Journal of Differential Equations. {\bf 217}, 1-13 (2017).

\bibitem{20}  Z. Li, Y. Liu, M. Yamamoto. Initial-boundary value problem for multi-term time-fractional diffusion equation with positive constant coefficients, Applied Mathematica and Computation. {\bf 257}, 381-397 (2015).

\bibitem{24}  S. Malik, S. Aziz. An inverse source problem for a two parameter anomalous diffusion equation with nonlocal boundary conditions, Computers and Mathematics with applications. 3, 7-19 (2017).

\bibitem{25}  M. Ruzhansky, N. Tokmagambetov, B.T. Torebek.  Inverse source problems for positive operators, I: Hypoelliptic diffusion and subdiffusion equations, J. Inverse Ill-Possed Probl. {\bf 27}, 891-911 (2019).

\bibitem{4} R.R. Ashurov, A.T. Mukhiddinova. Inverse Problem of Determining the Heat Source Density for the Subdiffusion Equation, Differential Equations. {\bf 56}:12, 1550-1563 (2020).

\bibitem{Kab1} S.I. Kabanikhin. Inverse and Ill-Posed Problems. Theory and Applications, De Gruyter (2011).

\bibitem{Pr} A.I. Prilepko, A.B. Kostin. On certain inverse problems for
parabolic equations with final and integral observation, Mat. Sb. { 183}:4, 49-68 (1992).

\bibitem{Sab}  K.B. Sabitov, A.R. Zaynullov. On the theory of the known inverse problems for the heat transfer equation, Series Physical and Mathematical Sciences. { 161}:2, 274-291 (2019).
\bibitem{Sab2} K.B. Sabitov, A.R. Zaynullov. Inverse problems for a two-dimensional heat equation with unknown right-hand side, Russian Math. 3, 75-88 (2021).

\bibitem{Orl} D.G. Orlovskii. On a problem of determining the parameter of an evolution equation, Differ. Uravn. {\bf 26}:9, 1614-1621 (1990).

\bibitem{Tix}  I.V. Tikhonov, Yu.S. \'Eidel’man. Uniqueness criterion in an inverse problem for an abstract differential equation with nonstationary inhomogeneous term, Mat. Notes. {\bf 77}:2, 273-290 (2005).
\bibitem{Tix2}  I.V. Tikhonov, Yu.S. \'Eidel’man. Problems of well-posedness of direct and inverse problems for an evolution equation of a special form, Mat. Zam. {\bf 56}:2, 99-113 (1994).

\bibitem{FN3} V.E. Fedorov, A. V. Urazaeva. An inverse problem for linear Sobolev type equations, Journal of Inverse and Ill-Posed Problems, 2004, V. 121, pp. 387-395. 

\bibitem{MS} M. Slodichka, Uniqueness for an inverse source problem of determining a space-dependent source in a non-autonomous time-fractional diffusion equation, Frac.Calculus and  Appl. Anal., 2020. V. 23, N 6, pp. 1703-1711. DOI:10.1515/fca-2020-0084.

\bibitem{MS1} M. Slodichka, K. Sishskova, V. Bockstal. Uniqueness for an inverse source problem of determining a space dependent source in a time-fractional diffusion equation, Appl. Math. Letters, 2019, V. 91, pp. 15-21.
\bibitem{FN} V.E. Fedorov, A.V. Nagumanova. Inverse problem for evolutionary equation with the Gerasimov–Caputo fractional derivative in the sectorial case, The Bulletin of Irkutsk State University. Series Mathematics, 2019, V. 28, pp. 123-137. 

\bibitem{FN2} V.E. Fedorov,  R.R. Nazhimov. Inverse problems for a class of degenerate evolution equations with Riemann – Liouville derivative, Frac.Cal., 2019, V. 22, pp. 271–286. 

\bibitem{AshM2} R.R. Ashurov, M.D. Shakarova. Inverse problem for the subdiffusion equation with fractional Caputo derivative, 16 november 2022, https://doi.org/10.48550/arXiv.2211.00081. 

\bibitem{Pskhu}  A.V. Pskhu, Fractional Differential Equations. Moscow: NAUKA. 2005 [in Russian].

\bibitem{AshM} R. Ashurov, A. Mukhiddinova. Initial-boundary value problem for a time-fractional subdiffusion equation with an arbitrary elliptic differential operator, Lobachevskii Journal of Mathematics. {\bf 42}:3, 517-525 (2021).
\bibitem{Dzh66} M. Dzherbashian M [=Djrbashian]. Integral Transforms and Representation of Functions in the Complex Domain, Moscow: NAUKA. 1966 (in Russian).

\bibitem{Gor} R. Gorenflo, A.A. Kilbas, F. Mainardi, S.V. Rogozin. Mittag-Leffler Functions, Related Topics and Applications, Springer.
Berlin-Heidelberg, Germany, 2014, doi: 10.1007/978-3-662-61550-8.
\end{thebibliography}
\end{document}